\documentclass[12pt,reqno]{amsart}

\usepackage{color}
\usepackage{amsmath, amsthm, amssymb}
\usepackage{accents}
\usepackage{amsfonts}
\usepackage[ansinew]{inputenc}
\usepackage[dvips]{epsfig}
\usepackage{graphicx}
\usepackage[english]{babel}
\usepackage{hyperref}
\theoremstyle{plain}
\newtheorem{thm}{Theorem}[section]

\newtheorem{lem}[thm]{Lemma}
\newtheorem{prop}[thm]{Proposition}

\theoremstyle{definition}
\newtheorem{defi}[thm]{Definition}

\theoremstyle{remark}
\newtheorem{rem}[thm]{Remark}

\numberwithin{equation}{section}

\newcommand{\average}{{\mathchoice {\kern1ex\vcenter{\hrule height.4pt
width 6pt depth0pt} \kern-9.7pt} {\kern1ex\vcenter{\hrule
height.4pt width 4.3pt depth0pt} \kern-7pt} {} {} }}

\def\R{\mathbb{R}}

\begin{document}

\title{The extremal solution for the fractional Laplacian}

\author{Xavier Ros-Oton}

\address{Universitat Polit\`ecnica de Catalunya, Departament de Matem\`{a}tica  Aplicada I, Diagonal 647, 08028 Barcelona, Spain}
\email{xavier.ros.oton@upc.edu}

\author{Joaquim Serra}

\address{Universitat Polit\`ecnica de Catalunya, Departament de Matem\`{a}tica  Aplicada I, Diagonal 647, 08028 Barcelona, Spain}

\email{joaquim.serra@upc.edu}

\thanks{The authors were supported by grants MINECO MTM2011-27739-C04-01 and GENCAT 2009SGR-345.}


\begin{abstract}
We study the extremal solution for the problem $(-\Delta)^s u=\lambda f(u)$ in $\Omega$,
$u\equiv0$ in $\R^n\setminus\Omega$, where $\lambda>0$ is a parameter and $s\in(0,1)$.
We extend some well known results for the extremal solution when the operator is the Laplacian to this nonlocal case.
For general convex nonlinearities we prove that the extremal solution is bounded in dimensions $n<4s$.
We also show that, for exponential and power-like nonlinearities, the extremal solution is bounded whenever $n<10s$.
In the limit $s\uparrow1$, $n<10$ is optimal.
In addition, we show that the extremal solution is $H^s(\R^n)$ in any dimension whenever the domain is convex.

To obtain some of these results we need $L^q$ estimates for solutions to the linear Dirichlet problem
for the fractional Laplacian with $L^p$ data.
We prove optimal $L^q$ and $C^\beta$ estimates, depending on the value of $p$.
These estimates follow from classical embedding results for the Riesz potential in $\R^n$.

Finally, to prove the $H^s$ regularity of the extremal solution we need an $L^\infty$ estimate near the boundary of convex domains,
which we obtain via the moving planes method.
For it, we use a maximum principle in small domains for integro-differential operators with decreasing kernels.
\end{abstract}

\maketitle

\tableofcontents

\section{Introduction and results}

Let $\Omega \subset\mathbb R^n$ be a bounded smooth domain and $s\in(0,1)$, and consider the problem
\begin{equation}\label{pb}
\left\{ \begin{array}{rcll} (-\Delta)^s u &=&\lambda f(u)&\textrm{in }\Omega \\
u&=&0&\textrm{in }\mathbb R^n\backslash\Omega,\end{array}\right.\end{equation}
where $\lambda$ is a positive parameter and $f:[0,\infty)\longrightarrow\mathbb R$ satisfies
\begin{equation}\label{condicions}f\textrm{ is } C^1\ \textrm{and nondecreasing},\ f(0)>0,\ \textrm{and}\ \lim_{t\rightarrow+\infty}\frac{f(t)}{t}=+\infty.\end{equation}
Here, $(-\Delta)^s$ is the fractional Laplacian, defined for $s\in(0,1)$ by
\begin{equation}
\label{laps}(-\Delta)^s u (x)= c_{n,s}{\rm PV}\int_{\R^n}\frac{u(x)-u(y)}{|x-y|^{n+2s}}dy,
\end{equation}
where $c_{n,s}$ is a constant.

It is well known ---see \cite{BV} or the excellent monograph \cite{D} and references therein--- that in the classical case $s=1$
there exists a finite extremal parameter $\lambda^*$ such that if $0<\lambda<\lambda^*$ then problem (\ref{pb})
admits a minimal classical solution $u_\lambda$, while for $\lambda>\lambda^*$ it has no solution, even in the weak sense.
Moreover, the family of functions $\{u_\lambda:0<\lambda<\lambda^*\}$ is increasing in $\lambda$, and its pointwise limit
$u^*=\lim_{\lambda\uparrow \lambda^*}u_\lambda$ is a weak solution of problem (\ref{pb}) with $\lambda=\lambda^*$.
It is called the extremal solution of (\ref{pb}).

When $f(u)=e^u$, we have that $u^*\in L^\infty(\Omega)$ if $n\leq9$ \cite{CR}, while $u^*(x)=\log\frac{1}{|x|^2}$ if $n\geq10$ and $\Omega=B_1$ \cite{JL}.
An analogous result holds for other nonlinearities such as powers $f(u)=(1+u)^p$ and also for functions $f$ satisfying a limit condition at infinity; see \cite{S}.
In the nineties H. Brezis and J.L. V\'azquez \cite{BV} raised the question of determining the regularity of $u^*$, depending on the dimension $n$, for general nonlinearities $f$ satisfying (\ref{condicions}).
The first result in this direction was proved by G. Nedev \cite{N}, who obtained that the extremal solution is bounded in dimensions $n\leq3$ whenever $f$ is convex.
Some years later, X. Cabr\'e and A. Capella \cite{CC} studied the radial case.
They showed that when $\Omega=B_1$ the extremal solution is bounded for all nonlinearities $f$ whenever $n\leq9$.
For general nonlinearities, the best known result at the moment is due to X. Cabr\'e \cite{C4}, and states that in dimensions $n\leq4$ then the extremal solution is bounded for any convex domain $\Omega$.
Recently, S. Villegas \cite{V} have proved, using the results in \cite{C4}, the boundedness of the extremal solution in dimension $n=4$ for all domains, not necessarily convex.
The problem is still open in dimensions $5\leq n\leq9$.

The aim of this paper is to study the extremal solution for the fractional Laplacian, that is, to study problem \eqref{pb} for $s\in(0,1)$.

The closest result to ours was obtained by Capella-D\'avila-Dupaigne-Sire \cite{CDDS}.
They studied the extremal solution in $\Omega=B_1$ for the {spectral} fractional Laplacian $A^s$.
The operator $A^s$, defined via the Dirichlet eigenvalues of the Laplacian in $\Omega$, is related to (but different from) the fractional Laplacian \eqref{laps}.
We will state their result later on in this introduction.

Let us start defining weak solutions to problem \eqref{pb}.

\begin{defi}\label{def} We say that $u\in L^1(\Omega)$ is a \emph{weak solution} of \eqref{pb} if
\begin{equation}\label{def1}
f(u)\delta^s\in L^1(\Omega),
\end{equation}
where $\delta(x)={\rm dist}(x,\partial\Omega)$, and
\begin{equation}\label{def2}
\int_\Omega u(-\Delta)^s\zeta dx=\int_\Omega\lambda f(u)\zeta dx
\end{equation}
for all $\zeta$ such that $\zeta$ and $(-\Delta)^s\zeta$ are bounded in $\Omega$ and $\zeta\equiv0$ on $\partial\Omega$.

Any bounded weak solution is a \emph{classical solution}, in the sense that it is regular in the interior of $\Omega$, continuous up to the boundary, and \eqref{pb} holds pointwise; see Remark \ref{solucions}.
\end{defi}

Note that for $s=1$ the above notion of weak solution is exactly the one used in \cite{BCMR, BV}.

In the classical case (that is, when $s=1$), the analysis of singular extremal solutions involves an intermediate class of solutions, those belonging to $H^1(\Omega)$; see \cite{BV,MP}.
These solutions are called \cite{BV} energy solutions.
As proved by Nedev \cite{N2}, when the domain $\Omega$ is convex the extremal solution belongs to $H^1(\Omega)$, and hence it is an energy solution; see \cite{CS} for the statement and proofs of the results in \cite{N2}.

Similarly, here we say that a weak solution $u$ is an \emph{energy solution} of \eqref{pb} when $u\in H^s(\R^n)$.
This is equivalent to saying that $u$ is a critical point of the
energy functional
\begin{equation}\label{energy}
\mathcal E(u)=\frac12\|u\|_{\accentset{\circ}{H}^s}^2-\int_\Omega \lambda F(u)dx,\qquad F'=f,
\end{equation}
where
\begin{equation}\label{Hs-norm}
\|u\|_{\accentset{\circ}{H}^s}^2=\int_{\R^n}\left|(-\Delta)^{s/2} u\right|^2dx=\frac{c_{n,s}}{2}\int_{\R^n}\int_{\R^n}\frac{|u(x)-u(y)|^2}{|x-y|^{n+2s}}dxdy=(u,u)_{\accentset{\circ}{H}^s}
\end{equation}
and
\begin{equation}\label{prod}
(u,v)_{\accentset{\circ}{H}^s}=\int_{\R^n}\hspace{-1mm}(-\Delta)^{s/2} u(-\Delta)^{s/2} v\,dx=\frac{c_{n,s}}{2}\hspace{-1mm}\int_{\R^n}\hspace{-1mm}\int_{\R^n} \hspace{-2mm} \frac{\bigl(u(x)-u(y)\bigr)\bigl(v(x)-v(y)\bigr)}{|x-y|^{n+2s}}dxdy.
\end{equation}

Our first result, stated next, concerns the existence of a minimal branch of solutions, $\{u_\lambda,\ 0<\lambda<\lambda^*\}$, with the same properties as in the case $s=1$.
These solutions are proved to be positive, bounded, increasing in $\lambda$, and semistable.
Recall that a weak solution $u$ of \eqref{pb} is said to be \emph{semistable} if
\begin{equation}\label{semistable}
\int_\Omega \lambda f'(u)\eta^2dx\leq \|\eta\|_{\accentset{\circ}{H}^s}^2
\end{equation}
for all $\eta\in H^s(\R^n)$ with $\eta\equiv0$ in $\R^n\setminus\Omega$.
When $u$ is an energy solution this is equivalent to saying that the second variation of energy $\mathcal E$ at $u$ is nonnegative.

\begin{prop}\label{existence} Let $\Omega\subset \mathbb R^n$ be a bounded smooth domain, $s\in(0,1)$, and $f$ be a function satisfying \eqref{condicions}.
Then, there exists a parameter $\lambda^*\in(0,\infty)$ such that:
\begin{itemize}
\item[(i)] If $0<\lambda<\lambda^*$, problem \eqref{pb} admits a minimal classical solution $u_\lambda$.
\item[(ii)] The family of functions $\{u_\lambda:0<\lambda<\lambda^*\}$ is increasing in $\lambda$, and its pointwise limit
    $u^*=\lim_{\lambda\uparrow \lambda^*}u_\lambda$ is a weak solution of \eqref{pb} with $\lambda=\lambda^*$.
\item[(iii)] For $\lambda>\lambda^*$, problem \eqref{pb} admits no classical solution.
\item[(iv)] These solutions $u_\lambda$, as well as $u^*$, are semistable.
\end{itemize}
\end{prop}

The weak solution $u^*$ is called the extremal solution of problem \eqref{pb}.

As explained above, the main question about the extremal solution $u^*$ is to decide whether it is bounded or not.
Once the extremal solution is bounded then it is a classical solution, in the sense that it satisfies equation \eqref{pb} pointwise.
For example, if $f\in C^\infty$ then $u^*$ bounded yields $u^*\in C^\infty(\Omega)\cap C^s(\overline\Omega)$.

Our main result, stated next, concerns  the regularity of the extremal solution for problem \eqref{pb}.
To our knowledge this is the first result concerning extremal solutions for \eqref{pb}.
In particular, the following are new results even for the unit ball $\Omega=B_1$ and for the exponential nonlinearity $f(u)=e^u$.

\begin{thm}\label{S} Let $\Omega$ be a bounded smooth domain in $\mathbb R^n$, $s\in(0,1)$, $f$ be a function satisfying \eqref{condicions}, and $u^*$ be the extremal solution of \eqref{pb}.
\begin{itemize}
\item[(i)] Assume that $f$ is convex. Then, $u^*$ is bounded whenever $n<4s$.
\item[(ii)] Assume that  $f$ is $C^2$ and that the following limit exists:
\begin{equation}\label{tau}
\tau:=\lim_{t\rightarrow+\infty}\frac{f(t)f''(t)}{f'(t)^2}.
\end{equation}
Then, $u^*$ is bounded whenever $n<10s$.
\item[(iii)] Assume that $\Omega$ is convex.
Then, $u^*$ belongs to $H^s(\R^n)$ for all $n\geq1$ and all $s\in(0,1)$.
\end{itemize}
\end{thm}

Note that the exponential and power nonlinearities $e^u$ and $(1+u)^p$, with $p>1$, satisfy the hypothesis in part (ii) whenever $n<10s$.
In the limit $s\uparrow1$, $n<10$ is optimal, since the extremal solution may be singular for $s=1$ and $n=10$ (as explained before in this introduction).

Note that the results in parts (i) and (ii) of Theorem \ref{S} do not provide any estimate when $s$ is small (more precisely, when $s\leq 1/4$ and $s\leq1/10$, respectively).
The boundedness of the extremal solution for small $s$ seems to require different methods from the ones that we present here.
Our computations in Section~\ref{sec-exp} suggest that the extremal solution for the fractional Laplacian should be bounded in dimensions $n\leq 7$ for all $s\in(0,1)$, at least for the exponential nonlinearity $f(u)=e^u$.
As commented above, Capella-D\'avila-Dupaigne-Sire \cite{CDDS} studied the extremal solution for the \emph{spectral} fractional Laplacian $A^s$ in $\Omega=B_1$.
They obtained an $L^\infty$ bound for the extremal solution in a ball in dimensions $n<2\left(2+s+\sqrt{2s+2}\right)$, and hence they proved the boundedness of the extremal solution in dimensions $n\leq 6$ for all $s\in(0,1)$.

To prove part (i) of  Theorem \ref{S} we borrow the ideas of \cite{N}, where Nedev proved the boundedness of the extremal solution for $s=1$ and $n\leq3$.
To prove part (ii) we follow the approach of M. Sanch\'on in \cite{S}.
When we try to repeat the same arguments for the fractional Laplacian, we find that some identities that in the case $s=1$ come from local integration by parts are no longer available for $s<1$.
Instead, we succeed to replace them by appropriate inequalities.
These inequalities are sharp as $s\uparrow1$, but not for small $s$.
Finally, part (iii) is proved by an argument of Nedev \cite{N2}, which for $s<1$ requires the Pohozaev identity for the fractional Laplacian, recently established by the authors in \cite{RS}.
This argument requires also some boundary estimates, which we prove using the moving planes method; see Proposition \ref{bdyestimates} at the end of this introduction.

An important tool in the proofs of the results of Nedev \cite{N} and Sanch\'on \cite{S} is the classical $L^p$ to $W^{2,p}$ estimate
for the Laplace equation.
Namely, if $u$ is the solution of $-\Delta u =g$ in $\Omega$, $u=0$ in $\partial\Omega$, with $g\in L^p(\Omega)$, $1<p<\infty$, then
\[\|u\|_{W^{2,p}(\Omega)}\leq C\|g\|_{L^p(\Omega)}.\]
This estimate and the Sobolev embeddings lead to $L^q(\Omega)$ or $C^\alpha(\overline\Omega)$ estimates for the solution $u$, depending on whether $1<p<\frac n2$ or $p>\frac n2$, respectively.

Here, to prove Theorem \ref{S} we need similar estimates but for the fractional Laplacian, in the sense that from $(-\Delta)^s u\in L^p(\Omega)$ we want to deduce $u\in L^q(\Omega)$ or $u\in C^\alpha(\overline\Omega)$.
However, $L^p$ to $W^{2s,p}$ estimates for the fractional Laplace equation, in which $-\Delta$ is replaced by the fractional Laplacian $(-\Delta)^s$, are not available for all $p$, even when $\Omega=\R^n$; see Remarks \ref{rem1} and \ref{rem11}.

Although the $L^p$ to $W^{2s,p}$ estimate does not hold for all $p$ in this fractional framework, what will be indeed true is the following result.
This is a crucial ingredient in the proof of Theorem \ref{S}.

\begin{prop}\label{th2}
Let $\Omega\subset \R^n$ be a bounded $C^{1,1}$ domain, $s\in(0,1)$, $n>2s$, $g\in C(\overline\Omega)$, and $u$ be the solution of
\begin{equation}\label{Omega}
\left\{ \begin{array}{rcll} (-\Delta)^s u &=&g&\textrm{in }\Omega \\
u&=&0&\textrm{in }\mathbb R^n\backslash\Omega.\end{array}\right.
\end{equation}
\begin{itemize}
\item[(i)] For each $1\leq r<\frac{n}{n-2s}$ there exists a constant $C$, depending only on $n$, $s$, $r$, and $|\Omega|$, such that
\[\|u\|_{L^r(\Omega)}\leq C\|g\|_{L^1(\Omega)},\quad  r<\frac{n}{n-2s}.\]
\item[(ii)] Let $1<p<\frac{n}{2s}$. Then there exists a constant $C$, depending only on $n$, $s$, and $p$, such that
\[\|u\|_{L^q(\Omega)}\leq C\|g\|_{L^p(\Omega)},\quad \textrm{where}\quad q=\frac{np}{n-2ps}.\]
\item[(iii)] Let $\frac{n}{2s}<p<\infty$. Then, there exists a constant $C$, depending only on $n$, $s$, $p$, and $\Omega$, such that
\[\|u\|_{C^{\beta}(\R^n)}\leq C\|g\|_{L^p(\Omega)},\quad \textrm{where}\quad \beta=\min\left\{s,2s-\frac{n}{p}\right\}.\]
\end{itemize}
\end{prop}

We will use parts (i), (ii), and (iii) of Proposition \ref{th2} in the proof of Theorem \ref{S}.
However, we will only use part (iii) to obtain an $L^\infty$ estimate for $u$, we will not need the $C^\beta$ bound.
Still, for completeness we prove the $C^\beta$ estimate, with the optimal exponent $\beta$ (depending on $p$).

\begin{rem}\label{n=1}
Proposition \ref{th2} does not provide any estimate for $n\leq 2s$.
Since $s\in(0,1)$, then $n\leq2s$ yields $n=1$ and $s\geq 1/2$.
In this case, any bounded domain is of the form $\Omega=(a,b)$, and the Green function $G(x,y)$ for problem \eqref{Omega} is explicit; see \cite{BGR}.
Then, by using this expression it is not difficult to show that $G(\cdot,y)$ is $L^\infty(\Omega)$ in case $s>1/2$ and $L^p(\Omega)$ for all $p<\infty$ in case $s=1/2$.
Hence, in case $n<2s$ it follows that $\|u\|_{L^\infty(\Omega)}\leq C\|g\|_{L^1(\Omega)}$, while in case $n=2s$ it follows that $\|u\|_{L^q(\Omega)}\leq C\|g\|_{L^1(\Omega)}$ for all $q<\infty$ and $\|u\|_{L^\infty(\Omega)}\leq C\|g\|_{L^p(\Omega)}$ for $p>1$.
\end{rem}

Proposition \ref{th2} follows from Theorem \ref{th1} and Proposition \ref{prop:u-is-Cbeta} below.
The first one contains some classical results concerning embeddings for the Riesz potential, and reads as follows.

\begin{thm}[see \cite{Stein}]\label{th1}
Let $s\in(0,1)$, $n>2s$, and $g$ and $u$ be such that
\begin{equation}\label{R^n}
u=(-\Delta)^{-s}g\ \ \textrm{in}\ \R^n,\end{equation}
in the sense that $u$ is the Riesz potential of order $2s$ of $g$.
Assume that $u$ and $g$ belong to $L^p(\R^n)$, with $1\leq p<\infty$.
\begin{itemize}
\item[(i)] If $p=1$, then there exists a constant $C$, depending only on $n$ and $s$, such that
\[\|u\|_{L^q_{{\rm weak}}(\R^n)}\leq C\|g\|_{L^1(\R^n)},\quad \textrm{where}\quad q=\frac{n}{n-2s}.\]

\item[(ii)] If $1<p<\frac{n}{2s}$, then there exists a constant $C$, depending only on $n$, $s$, and $p$, such that
\[\|u\|_{L^q(\R^n)}\leq C\|g\|_{L^p(\R^n)},\quad \textrm{where}\quad q=\frac{np}{n-2ps}.\]

\item[(iii)] If $\frac{n}{2s}<p<\infty$, then there exists a constant $C$, depending only on $n$, $s$, and $p$, such that
\[[u]_{C^{\alpha}(\R^n)}\leq C\|g\|_{L^p(\R^n)},\quad \textrm{where}\quad \alpha=2s-\frac{n}{p},\]
where $[\,\cdot\,]_{C^{\alpha}(\R^n)}$ denotes the $C^\alpha$ seminorm.
\end{itemize}
\end{thm}

Parts (i) and (ii) of Theorem \ref{th1} are proved in the book of Stein \cite[Chapter V]{Stein}.
Part (iii) is also a classical result, but it seems to be more difficult to find an exact reference for it.
Although it is not explicitly stated in \cite{Stein}, it follows for example from the inclusions
\[I_{2s}(L^p)=I_{2s-n/p}(I_{n/p}(L^p))\subset I_{2s-n/p}(\textrm{BMO})\subset C^{2s-\frac{n}{p}},\]
which are commented in \cite[p.164]{Stein}.
In the more general framework of spaces with non-doubling $n$-dimensional measures, a short proof of this result can also be found in \cite{GG}.

Having Theorem \ref{th1} available, to prove Proposition \ref{th2} we will argue as follows.
Assume $1< p<\frac{n}{2s}$ and consider the solution $v$ of the problem
\[(-\Delta)^s v=|g|\ \ {\rm in}\ \R^n,\]
where $g$ is extended by zero outside $\Omega$.
On the one hand, the maximum principle yields $-v\leq u\leq v$ in $\R^n$, and by Theorem \ref{th1} we have that $v\in L^q(\R^n)$.
From this, parts (i) and (ii) of the proposition follow.
On the other hand, if $p>\frac{n}{2s}$ we write $u=\tilde v+w$, where $\tilde v$ solves $(-\Delta)^s \tilde v=g$ in $\R^n$ and $w$ is the solution of
\[\left\{ \begin{array}{rcll} (-\Delta)^s w &=&0&\textrm{in }\Omega \\
w&=&\tilde v&\textrm{in }\mathbb R^n\backslash\Omega.\end{array}\right.\]
As before, by Theorem \ref{th1} we will have that $\tilde v\in C^\alpha(\R^n)$, where $\alpha=2s-\frac{n}{p}$.
Then, the $C^\beta$ regularity of $u$ will follow from the following new result.

\begin{prop}\label{prop:u-is-Cbeta}
Let $\Omega$ be a bounded $C^{1,1}$ domain, $s\in(0,1)$, $h\in C^\alpha(\R^n\setminus\Omega)$ for some $\alpha>0$, and $u$ be the solution of
\begin{equation}\label{g}
\left\{ \begin{array}{rcll} (-\Delta)^s u &=&0&\textrm{in }\Omega \\
u&=&h&\textrm{in }\mathbb R^n\backslash\Omega.\end{array}\right.
\end{equation}
Then, $u\in C^\beta(\R^n)$, with $\beta=\min\{s,\alpha\}$, and
\[\|u\|_{C^{\beta}(\R^n)}\le C \|h\|_{C^\alpha(\R^n\setminus\Omega)},\]
where $C$ is a constant depending only on $\Omega$, $\alpha$, and $s$.
\end{prop}

To prove Proposition \ref{prop:u-is-Cbeta} we use similar ideas as in \cite{RS-Dir}.
Namely, since $u$ is harmonic then it is smooth inside $\Omega$.
Hence, we only have to prove $C^\beta$ estimates near the boundary.
To do it, we use an appropriate barrier to show that
\[|u(x)-u(x_0)|\leq C\|h\|_{C^{\alpha}}\delta(x)^\beta\quad {\rm in}\ \Omega,\]
where $x_0$ is the nearest point to $x$ on $\partial\Omega$, $\delta(x)={\rm dist}(x,\partial\Omega)$, and $\beta=\min\{s,\alpha\}$.
Combining this with the interior estimates, we obtain $C^\beta$ estimates up to the boundary of $\Omega$.

Finally, as explained before, to show that when the domain is convex the extremal solution belongs to the energy class $H^s(\R^n)$ ---which is part (iii) of Theorem \ref{S}--- we need the following boundary estimates.

\begin{prop}\label{bdyestimates}
Let $\Omega\subset\R^n$ be a bounded convex domain, $s\in(0,1)$, $f$ be a locally Lipschitz function, and $u$ be a bounded positive solution of
\begin{equation}\label{Omega}
\left\{ \begin{array}{rcll} (-\Delta)^s u &=&f(u)&\textrm{in }\Omega \\
u&=&0&\textrm{in }\mathbb R^n\backslash\Omega.\end{array}\right.
\end{equation}
Then, there exists constants $\delta>0$ and $C$, depending only on $\Omega$, such that
\[\|u\|_{L^{\infty}(\Omega_\delta)}\leq C\|u\|_{L^1(\Omega)},\]
where $\Omega_\delta=\{x\in\Omega\,:\, {\rm dist}(x,\partial\Omega)<\delta\}$.
\end{prop}

This estimate follows, as in the classical result of de Figueiredo-Lions-Nussbaum \cite{FLN}, from the moving planes method.
There are different versions of the moving planes method for the fractional Laplacian (using the Caffarelli-Silvestre extension, the Riesz potential, the Hopf lemma, etc.).
A particularly clean version uses the maximum principle in small domains for the fractional Laplacian, recently proved by Jarohs and Weth in \cite{JW}.
Here, we follow their approach and we show that this maximum principle holds also for integro-differential operators with decreasing kernels.

The paper is organized as follows.
In Section \ref{sec-exist} we prove Proposition \ref{existence}.
In Section \ref{sec-exp} we study the regularity of the extremal solution in the case $f(u)=e^u$.
In Section \ref{sec-reg} we prove Theorem \ref{S} (i)-(ii).
In Section \ref{sec-bdry} we show the maximum principle in small domains and use the moving planes method to establish Proposition \ref{bdyestimates}.
In Section \ref{sec-Hs} we prove Theorem \ref{S} (iii).
Finally, in Section \ref{sec3} we prove Proposition \ref{th2}.

\section{Existence of the extremal solution}
\label{sec-exist}

In this section we prove Proposition \ref{existence}.
For it, we follow the argument from Proposition 5.1 in \cite{CC}; see also \cite{D}.

\begin{proof}[Proof of Proposition \ref{existence}]
\emph{Step 1.} We first prove that there is no weak solution for large $\lambda$.

Let $\lambda_1>0$ be the first eigenvalue of $(-\Delta)^s$ in $\Omega$ and $\varphi_1>0$ the corresponding eigenfunction, that is,
\[\left\{\begin{array}{rcll}
(-\Delta)^s \varphi_1 &=&\lambda_1\varphi_1&\ \textrm{in}\ \Omega\\
\varphi_1&>&0&\ \textrm{in}\ \Omega\\
\varphi_1&=&0&\ \textrm{in}\ \R^n\setminus\Omega.
\end{array}\right.\]
The existence, simplicity, and boundedness of the first eigenfunction is proved in \cite[Proposition 5]{SV2} and \cite[Proposition 4]{SV3}.
Assume that $u$ is a weak solution of \eqref{pb}.
Then, using $\varphi_1$ as a test function for problem \eqref{pb} (see Definition \ref{def}), we obtain
\begin{equation}\label{contradiction}
\int_\Omega \lambda_1u\,\varphi_1 dx=\int_\Omega u(-\Delta)^s\varphi_1 dx=\int_\Omega\lambda f(u)\varphi_1 dx.
\end{equation}
But since $f$ is superlinear at infinity and positive in $[0,\infty)$, it follows that $\lambda f(u)>\lambda_1 u$ if $\lambda$ is large enough, a contradiction with \eqref{contradiction}.

\emph{Step 2.} Next we prove the existence of a classical solution to \eqref{pb} for small $\lambda$.
Since $f(0)>0$, $\underline{u}\equiv0$ is a strict subsolution of \eqref{pb} for every $\lambda>0$.
The solution $\overline u$ of
\begin{equation}\label{f=1}
\left\{ \begin{array}{rcll} (-\Delta)^s \overline u &=&1&\textrm{in }\Omega \\
\overline u&=&0&\textrm{on }\mathbb R^n\backslash\Omega\end{array}\right.\end{equation}
is a bounded supersolution of \eqref{pb} for small $\lambda$, more precisely whenever $\lambda f(\max \overline u)<1$.
For such values of $\lambda$, a classical solution $u_\lambda$ is obtained by monotone iteration starting from zero; see for example \cite{D}.

\emph{Step 3.} We next prove that there exists a finite parameter $\lambda^*$ such that for $\lambda<\lambda^*$ there is a classical solution while for $\lambda>\lambda^*$ there does not exist classical solution.

Define $\lambda^*$ as the supremum of all $\lambda>0$ for which \eqref{pb} admits a classical solution.
By Steps 1 and 2, it follows that $0<\lambda^*<\infty$.
Now, for each $\lambda<\lambda^*$ there exists $\mu\in (\lambda,\lambda^*)$ such that \eqref{pb} admits a classical solution $u_\mu$.
Since $f>0$, $u_\mu$ is a bounded supersolution of \eqref{pb}, and hence the monotone iteration procedure shows that \eqref{pb} admits a classical solution $u_\lambda$ with $u_\lambda\leq u_\mu$.
Note that the iteration procedure, and hence the solution that it produces, are independent of the supersolution
$u_\mu$.
In addition, by the same reason $u_\lambda$ is smaller than any bounded supersolution of \eqref{pb}.
It follows that $u_\lambda$ is minimal (i.e., the smallest solution) and that $u_\lambda<u_\mu$.

\emph{Step 4.} We show now that these minimal solutions $u_\lambda$, $0<\lambda<\lambda^*$, are semistable.

Note that the energy functional \eqref{energy} for problem \eqref{pb} in the set $\{u\in H^s(\R^n)\,:\, u\equiv0\ \textrm{in}\ \R^n\setminus\Omega,\ 0\leq u\leq u_\lambda\}$ admits an absolute minimizer $u_{\textrm{min}}$.
Then, using that $u_\lambda$ is the minimal solution and that $f$ is positive and increasing, it is not difficult to see that $u_{\textrm{min}}$ must coincide with $u_\lambda$.
Considering the second variation of energy (with respect to nonpositive perturbations) we see that $u_{\textrm{min}}$ is a semistable solution of \eqref{pb}.
But since $u_{\textrm{min}}$ agrees with $u_\lambda$, then $u_\lambda$ is semistable.
Thus $u_\lambda$ is semistable.

\emph{Step 5.} We now prove that the pointwise limit $u^*=\lim_{\lambda\uparrow\lambda^*} u_\lambda$ is a weak solution of \eqref{pb} for $\lambda=\lambda^*$ and that this solution $u^*$ is semistable.

As above, let $\lambda_1>0$ the first eigenvalue of $(-\Delta)^s$, and $\varphi_1>0$ be the corresponding eigenfunction.
Since $f$ is superlinear at infinity, there exists a constant $C>0$ such that
\begin{equation}\label{ineq-f}
\frac{2\lambda_1}{\lambda^*}t\leq f(t)+C\quad \textrm{for all}\quad t\geq0.
\end{equation}
Using $\varphi_1$ as a test function in \eqref{def2} for $u_\lambda$, we find
\[\int_\Omega \lambda f(u_\lambda)\varphi_1dx=\int_\Omega \lambda_1 u_\lambda \varphi_1dx\leq \frac{\lambda^*}{2}\int_\Omega\left( f(u_\lambda)+C\right)\varphi_1dx.\]
In the last inequality we have used \eqref{ineq-f}.
Taking $\lambda\geq \frac34 \lambda^*$, we see that $f(u_\lambda)\varphi_1$ is uniformly bounded in $L^1(\Omega)$.
In addition, it follows from the results in \cite{RS-Dir} that
\[c_1\delta^s\leq \varphi_1\leq C_2\delta^s\ \ {\rm in}\ \Omega\]
for some positive constants $c_1$ and $C_2$, where $\delta(x)={\rm dist}(x,\partial\Omega)$.
Hence, we have that
\[\lambda\int_\Omega f(u_\lambda)\delta^sdx\leq C\]
for some constant $C$ that does not depend on $\lambda$.
Use now $\overline u$, the solution of \eqref{f=1}, as a test function.
We obtain that
\[\int_\Omega u_\lambda dx=\lambda\int_\Omega f(u_\lambda)\overline u dx\leq C_3\lambda\int_\Omega f(u_\lambda)\delta^sdx\leq C\]
for some constant $C$ depending only on $f$ and $\Omega$.
Here we have used that $\overline u\leq C_3\delta^s$ in $\Omega$ for some constant $C_3>0$, which also follows from \cite{RS-Dir}.

Thus, both sequences, $u_\lambda$ and $\lambda f(u_\lambda)\delta^s$ are increasing in $\lambda$ and uniformly bounded in $L^1(\Omega)$ for $\lambda<\lambda^*$.
By monotone convergence, we conclude that $u^*\in L^1(\Omega)$ is a weak solution of \eqref{pb} for $\lambda=\lambda^*$.

Finally, for $\lambda<\lambda^*$ we have $\int_\Omega \lambda f'(u_\lambda)|\eta|^2dx\leq \|\eta\|_{\accentset{\circ}{H}^s}^2$, where $\|\eta\|_{\accentset{\circ}{H}^s}^2$ is defined by \eqref{Hs-norm}, for all $\eta\in H^s(\R^n)$ with $\eta\equiv0$ in $\R^n\setminus\Omega$.
Since $f'\geq0$, Fatou's lemma leads to
\[\int_\Omega \lambda^* f'(u^*)|\eta|^2dx\leq \|\eta\|_{\accentset{\circ}{H}^s}^2,\]
and hence $u^*$ is semistable.
\end{proof}

\begin{rem}\label{solucions}
As said in the introduction, the study of extremal solutions involves three classes of solutions: classical, energy, and weak solutions; see Definition \ref{def}.
It follows from their definitions that any classical solution is an energy solution, and that any energy solution is a weak solution.

Moreover, any weak solution $u$ which is bounded is a classical solution.
This can be seen as follows.
First, by considering $u\ast\eta^\epsilon$ and $f(u)\ast\eta^\epsilon$, where $\eta^\epsilon$ is a standard mollifier, it is not difficult to see that $u$ is regular in the interior of $\Omega$.
Moreover, by scaling, we find that $|(-\Delta)^{s/2}u|\leq C\delta^{-s}$, where $\delta(x)=\textrm{dist}(x,\partial\Omega)$.
Then, if $\zeta\in C^\infty_c(\Omega)$, we can integrate by parts in \eqref{def2} to obtain
\begin{equation}\label{eq-rem}
(u,\zeta)_{\accentset{\circ}{H}^s}=\int_{\R^n}\int_{\R^n}\frac{\bigl(u(x)-u(y)\bigr)\bigl(\zeta(x)-\zeta(y)\bigr)}{|x-y|^{n+2s}}dx\,dy=\int_\Omega \lambda f(u)\zeta dx
\end{equation}
for all $\zeta\in C^\infty_c(\Omega)$.
Hence, since $f(u)\in L^\infty$, by density \eqref{eq-rem} holds for all $\zeta\in H^s(\R^n)$ such that $\zeta\equiv0$ in $\R^n\setminus\Omega$, and therefore $u$ is an energy solution.
Finally, bounded energy solutions are classical solutions; see Remark 2.11 in \cite{RS-Dir} and \cite{SV}.
\end{rem}

\section{An example case: the exponential nonlinearity}
\label{sec-exp}

In this section we study the regularity of the extremal solution for the nonlinearity $f(u)=e^u$.
Although the results of this section follow from Theorem \ref{S} (ii), we exhibit this case separately because the proofs are much simpler.
Furthermore, this exponential case has the advantage that we have an explicit unbounded solution to the equation in the whole $\R^n$, and we can compute the values of $n$ and $s$ for which this singular solution is semistable.

The main result of this section is the following.

\begin{prop}\label{exp} Let $\Omega$ be a smooth and bounded domain in $\mathbb R^n$, and let $u^*$ the extremal solution of \eqref{pb}.
Assume that $f(u)=e^u$ and $n<10s$.
Then, $u^*$ is bounded.
\end{prop}

\begin{proof} Let $\alpha$ be a positive number to be chosen later.
Setting $\eta=e^{\alpha u_\lambda}-1$ in the stability condition (\ref{semistable}) (note that $\eta\equiv0$ in $\R^n\setminus\Omega$), we obtain that
\begin{equation}\label{exp1}
\int_\Omega \lambda e^{u_\lambda}(e^{\alpha u_\lambda}-1)^2dx\leq \left\|e^{\alpha u_\lambda}-1\right\|_{\accentset{\circ}{H}^s}^2.\end{equation}

Next we use that
\begin{equation}\label{ineqexp}
\left(e^{b}-e^{a}\right)^2\leq \frac{1}{2}\left(e^{2 b}-e^{2 a}\right)(b-a)\end{equation}
for all real numbers $a$ and $b$.
This inequality can be deduced easily from the Cauchy-Schwarz inequality, as follows
\[\left(e^{b}-e^{a}\right)^2=\left(\int_a^b e^{t}dt\right)^2\leq (b-a)\int_a^b e^{2 t}dt=\frac{1}{2}\left(e^{2 b}-e^{2 a}\right)(b-a).\]
Using \eqref{ineqexp}, \eqref{prod}, and integrating by parts, we deduce
\begin{eqnarray*}
\left\|e^{\alpha u_\lambda}-1\right\|_{\accentset{\circ}{H}^s}^2&=&\frac{c_{n,s}}{2}\int_{\R^n}\int_{\R^n}\frac{\left(e^{\alpha u_\lambda(x)} - e^{\alpha u_\lambda(y)}\right)^2}{|x-y|^{n+2s}}dxdy\\
&\leq& \frac{c_{n,s}}{2}\int_{\R^n}\int_{\R^n}\frac{\frac{1}{2}\left(e^{2\alpha u_\lambda(x)} - e^{2\alpha u_\lambda(y)}\right)\left(\alpha u_\lambda(x)-\alpha u_\lambda(y)\right)}{|x-y|^{n+2s}}dxdy\\
&=&\frac{\alpha}{2}\int_{\Omega} e^{2\alpha u_\lambda}(-\Delta)^su_\lambda dx.
\end{eqnarray*}
Thus, using that $(-\Delta)^s u_\lambda=\lambda e^{u_\lambda}$, we find
\begin{equation}\label{exp2}
\left\|e^{\alpha u_\lambda}-1\right\|_{\accentset{\circ}{H}^s}^2\leq \frac{\alpha}{2}\int_{\Omega} e^{2\alpha u_\lambda}(-\Delta)^su_\lambda dx=\frac{\alpha}{2}\int_\Omega \lambda e^{(2\alpha+1)u_\lambda}dx.\end{equation}

Therefore, combining \eqref{exp1} and \eqref{exp2}, and rearranging terms, we get
\[\left(1-\frac\alpha2\right)\int_\Omega e^{(2\alpha+1)u_\lambda}-2\int_\Omega e^{(\alpha+1)u_\lambda}+\int_\Omega e^{\alpha u_\lambda}\leq0.\]
From this, it follows from H\"older's inequality that for each $\alpha<2$
\begin{equation}\label{L^5}
\|e^{u_\lambda}\|_{L^{2\alpha+1}}\leq C
\end{equation}
for some constant $C$ which depends only on $\alpha$ and $|\Omega|$.

Finally, given $n<10s$ we can choose $\alpha<2$ such that $\frac{n}{2s}<2\alpha+1<5$.
Then, taking $p=2\alpha+1$ in Proposition \ref{th2} (iii) (see also Remark \ref{n=1}) and using \eqref{L^5} we obtain
\[\|u_\lambda\|_{L^\infty(\Omega)}\leq C_1\|(-\Delta)^s u_\lambda\|_{L^p(\Omega)}= C_1\lambda\|e^{u_\lambda}\|_{L^p(\Omega)}\leq C\]
for some constant $C$ that depends only on $n$, $s$, and $\Omega$.
Letting $\lambda\uparrow \lambda^*$ we find that the extremal solution $u^*$ is bounded, as desired.
\end{proof}

The following result concerns the stability of the explicit singular solution $\log\frac{1}{|x|^{2s}}$ to equation $(-\Delta)^s u=\lambda e^u$ in the whole $\R^n$.

\begin{prop}\label{stability-exp} Let $s\in(0,1)$, and let
\[u_0(x)=\log\frac{1}{|x|^{2s}}.\]
Then, $u_0$ is a solution of $(-\Delta)^su=\lambda_0e^{u}$ in all of $\R^n$ for some $\lambda_0>0$.
Moreover, $u_0$ is semistable if and only if
\begin{equation}\label{semistable-condition-Gammas}
\frac{\Gamma\left(\frac{n}{2}\right)\Gamma(1+s)}{\Gamma\left(\frac{n-2s}{2}\right)}\leq \frac{\Gamma^2\left(\frac{n+2s}{4}\right)}{\Gamma^2\left(\frac{n-2s}{4}\right)}.
\end{equation}
As a consequence:
\begin{itemize}
\item If $n\leq7$, then $u$ is unstable for all $s\in(0,1)$.
\item If $n=8$, then $u$ is semistable if and only if $s\lesssim0'28206...$.
\item If $n=9$, then $u$ is semistable if and only if $s\lesssim0'63237...$.
\item If $n\geq10$, then $u$ is semistable for all $s\in(0,1)$.
\end{itemize}
\end{prop}

Proposition \ref{stability-exp} suggests that the extremal solution for the fractional Laplacian should be bounded whenever
\begin{equation}\label{ineq-gammas}
\frac{\Gamma\left(\frac{n}{2}\right)\Gamma(1+s)}{\Gamma\left(\frac{n-2s}{2}\right)}> \frac{\Gamma^2\left(\frac{n+2s}{4}\right)}{\Gamma^2\left(\frac{n-2s}{4}\right)},
\end{equation}
at least for the exponential nonlinearity $f(u)=e^u$.
In particular, $u^*$ should be bounded for all $s\in(0,1)$ whenever $n\leq7$.
This is an open problem.

\begin{rem}
When $s=1$ and when $s=2$, inequality \eqref{ineq-gammas} coincides with the expected optimal dimensions for which the extremal solution is bounded for the Laplacian $\Delta$ and for the bilaplacian $\Delta^2$, respectively.
In the unit ball $\Omega=B_1$, it is well known that the extremal solution for $s=1$ is bounded whenever $n\leq9$ and may be singular if $n\geq10$ \cite{CC}, while the extremal solution for $s=2$ is bounded whenever $n\leq12$ and may be singular if $n\geq13$ \cite{DDGM}.
Taking $s=1$ and $s=2$ in \eqref{ineq-gammas}, one can see that the inequality is equivalent to $n<10$ and $n\lesssim12.5653...$, respectively.
\end{rem}

We next give the

\begin{proof}[Proof of Proposition \ref{stability-exp}]
First, using the Fourier transform, it is not difficult to compute
\[(-\Delta)^su_0=(-\Delta)^s\log\frac{1}{|x|^{2s}}=\frac{\lambda_0}{|x|^{2s}},\]
where
\[\lambda_0=2^{2s}\frac{\Gamma\left(\frac{n}{2}\right)\Gamma(1+s)}{\Gamma\left(\frac{n-2s}{2}\right)}.\]
Thus, $u_0$ is a solution of $(-\Delta)^su_0=\lambda_0 e^{u_0}$.

Now, since $f(u)=e^u$, by \eqref{semistable} we have that $u_0$ is semistable in $\Omega=\R^n$ if and only if
\[\lambda_0\int_{\R^n} \frac{\eta^2}{|x|^{2s}}dx\leq \int_{\R^n}\left|(-\Delta)^{s/2}\eta\right|^2dx\]
for all $\eta\in H^s(\R^n)$.

The inequality
\[\int_\Omega \frac{\eta^2}{|x|^{2s}}dx\leq H_{n,s}^{-1}\int_{\R^n}\left|(-\Delta)^{s/2}\eta\right|^2dx\]
is known as the fractional Hardy inequality, and the best constant
\[H_{n,s}=2^{2s}\frac{\Gamma^2\left(\frac{n+2s}{4}\right)}{\Gamma^2\left(\frac{n-2s}{4}\right)}\]
was obtained by Herbst \cite{H} in 1977; see also \cite{FLS}.
Therefore, it follows that $u_0$ is semistable if and only if
\[\lambda_0\leq H_{n,s},\]
which is the same as \eqref{semistable-condition-Gammas}.
\end{proof}

\section{Boundedness of the extremal solution in low dimensions}
\label{sec-reg}

In this section we prove Theorem \ref{S} (i)-(ii).

We start with a lemma, which is the generalization of inequality \eqref{ineqexp}.
It will be used in the proof of both parts (i) and (ii) of Theorem \ref{S}.

\begin{lem}\label{po} Let $f$ be a $C^1([0,\infty))$ function, $\widetilde f(t)=f(t)-f(0)$, $\gamma>0$, and
\begin{equation}\label{functiong}
g(t)=\int_0^t \widetilde f(s)^{2\gamma-2}f'(s)^2ds.
\end{equation}
Then,
\[\left(\widetilde f(a)^\gamma-\widetilde f(b)^\gamma\right)^2\leq \gamma^2\bigl(g(a)-g(b)\bigr)(a-b)\]
for all nonnegative numbers $a$ and $b$.
\end{lem}

\begin{proof} We can assume $a\leq b$.
Then, since $\frac{d}{dt}\left\{\widetilde f(t)^\gamma\right\}=\gamma \widetilde f(t)^{\gamma-1}f'(t)$,
the inequality can be written as
\[\left(\int_a^b\gamma \widetilde f(t)^{\gamma-1}f'(t)dt\right)^2\leq \gamma^2(b-a)\int_a^b \widetilde f(t)^{2\gamma-2}f'(t)^2dt,\]
which follows from the Cauchy-Schwarz inequality.
\end{proof}

The proof of part (ii) of Theorem \ref{S} will be split in two cases.
Namely, $\tau\geq1$ and $\tau<1$, where $\tau$ is given by \eqref{tau}.
For the case $\tau\geq1$, Lemma \ref{lemamanel} below will be an important tool.
Instead, for the case $\tau<1$ we will use Lemma \ref{lemapotencia}.
Both lemmas are proved by Sanch\'on in \cite{S}, where the extremal solution for the $p$-Laplacian operator is studied.

\begin{lem}[\cite{S}]\label{lemamanel}
Let $f$ be a function satisfying \eqref{condicions}, and assume that the limit in \eqref{tau} exists.
Assume in addition that
\[\tau=\lim_{t\rightarrow\infty}\frac{f(t)f''(t)}{f'(t)^2}\geq1.\]
Then, any $\gamma\in(1,1+\sqrt{\tau})$ satisfies
\begin{equation}\label{condgamma}
\limsup_{t\rightarrow+\infty}\frac{\gamma^2g(t)}{f(t)^{2\gamma-1}f'(t)}<1,
\end{equation}
where $g$ is given by \eqref{functiong}.
\end{lem}

\begin{lem}[\cite{S}]\label{lemapotencia}
Let $f$ be a function satisfying \eqref{condicions}, and assume that the limit in \eqref{tau} exists.
Assume in addition that
\[\tau=\lim_{t\rightarrow\infty}\frac{f(t)f''(t)}{f'(t)^2}<1.\]
Then, for every $\epsilon\in(0,1-\tau)$ there exists a positive constant $C$ such that
\[f(t)\leq C(1+t)^{\frac{1}{1-(\tau+\epsilon)}},\qquad \mbox{for all}\ \ t>0.\]
The constant $C$ depends only on $\tau$ and $\epsilon$.
\end{lem}

The first step in the proof of Theorem \ref{S} (ii) in case $\tau\geq1$ is the following result.

\begin{lem}\label{propgamma} Let $f$ be a function satisfying \eqref{condicions}.
Assume that $\gamma\geq1$ satisfies \eqref{condgamma},
where $g$ is given by \eqref{functiong}.
Let $u_\lambda$ be the solution of \eqref{pb} given by Proposition \ref{existence}~(i), where $\lambda<\lambda^*$.
Then,
\[\|f(u_\lambda)^{2\gamma}f'(u_\lambda)\|_{L^1(\Omega)}\leq C\]
for some constant $C$ which does not depend on $\lambda$.
\end{lem}

\begin{proof}
Recall that the seminorm $\|\,\cdot\,\|_{\accentset{\circ}{H}^s}$ is defined by \eqref{Hs-norm}.
Using Lemma \ref{po}, \eqref{prod}, and integrating by parts,
\begin{equation}\begin{split}\label{fr}
\left\|\widetilde f(u_\lambda)^\gamma\right\|_{\accentset{\circ}{H}^s}^2&=\frac{c_{n,s}}{2}\int_{\R^n}\int_{\R^n}\frac{\left(\widetilde f(u_\lambda(x))^\gamma - \widetilde f(u_\lambda(y))^\gamma\right)^2}{|x-y|^{n+2s}}dxdy\\
&\leq \gamma^2 \frac{c_{n,s}}{2}\int_{\R^n}\int_{\R^n}\frac{\bigl(g(u_\lambda(x)) - g( u_\lambda(y))\bigr)\left(u_\lambda(x)-u_\lambda(y)\right)}{|x-y|^{n+2s}}dxdy\\
&= \gamma^2 \int_{\R^n}(-\Delta)^{s/2}g(u_\lambda)(-\Delta)^{s/2}u_\lambda\,dx\\
&=\gamma^2\int_{\Omega} g(u_\lambda)(-\Delta)^su_\lambda\,dx\\
&=\gamma^2\int_{\Omega} f(u_\lambda)g(u_\lambda)dx.\end{split}
\end{equation}

Moreover, the stability condition \eqref{semistable} applied with $\eta=\widetilde f(u_\lambda)^\gamma$ yields
\[\int_\Omega f'(u_\lambda)\widetilde f(u_\lambda)^{2\gamma} \leq \left\|\widetilde f(u_\lambda)^\gamma\right\|_{\accentset{\circ}{H}^s}^2.\]
This, combined with \eqref{fr}, gives
\begin{equation}\label{4nedev4}
\int_\Omega f'(u_\lambda)\widetilde f(u_\lambda)^{2\gamma}\leq \gamma^2\int_\Omega f(u_\lambda)g(u_\lambda).\end{equation}

Finally, by \eqref{condgamma} and since $\widetilde f(t)/f(t)\rightarrow1$ as $t\rightarrow+\infty$, it follows from \eqref{4nedev4} that
\begin{equation}\label{h1}
\int_\Omega f(u_\lambda)^{2\gamma}f'(u_\lambda)\leq C
\end{equation}
for some constant $C$ that does not depend on $\lambda$, and thus the proposition is proved.
\end{proof}

We next give the proof of Theorem \ref{S} (ii).

\begin{proof}[Proof of Theorem \ref{S} (ii)]
Assume first that $\tau\geq1$, where
\[\tau=\lim_{t\rightarrow\infty}\frac{f(t)f''(t)}{f'(t)^2}.\]
By Lemma \ref{propgamma} and Lemma \ref{lemamanel}, we have that
\begin{equation}\label{visL^1}
\int_{\Omega}f(u_\lambda)^{2\gamma}f'(u_\lambda)dx\leq C\end{equation}
for each $\gamma\in(1,1+\sqrt{\tau})$.

Now, for any such $\gamma$, we have that ${\widetilde f}^{2\gamma}$ is increasing and convex (since $2\gamma\geq1$), and thus
\[\widetilde f(a)^{2\gamma}-\widetilde f(b)^{2\gamma}\leq 2\gamma f'(a)\widetilde f(a)^{2\gamma-1}(a-b).\]
Therefore, we have that
\begin{eqnarray*}(-\Delta)^s\widetilde f(u_\lambda)^{2\gamma}(x) &=& c_{n,s}\int_{\R^n}\frac{\widetilde f(u_\lambda(x))^{2\gamma}-\widetilde f(u_\lambda(y))^{2\gamma}}{|x-y|^{n+2s}}dy\\
&\leq& 2\gamma f'(u_\lambda(x))\widetilde f(u_\lambda(x))^{2\gamma-1}c_{n,s}\int_{\R^n}\frac{u_\lambda(x)-u_\lambda(y)}{|x-y|^{n+2s}}dy\\
&=&2\gamma f'(u_\lambda(x))\widetilde f(u_\lambda(x))^{2\gamma-1}(-\Delta)^s u_\lambda(x)\\
&\leq&2\gamma\lambda f'(u_\lambda(x))f(u_\lambda(x))^{2\gamma},\end{eqnarray*}
and thus,
\begin{equation}\label{defv}
(-\Delta)^s\widetilde f(u_\lambda)^{2\gamma}\leq 2\gamma \lambda f'(u_\lambda)f(u_\lambda)^{2\gamma}:=v(x).
\end{equation}

Let now $w$ be the solution of the problem
\begin{equation}\label{hv2}
\left\{ \begin{array}{rcll} (-\Delta)^s w &=&v&\textrm{in }\Omega \\
w&=&0&\textrm{in }\mathbb R^n\backslash\Omega,
\end{array}\right.\end{equation}
where $v$ is given by \eqref{defv}.
Then, by \eqref{visL^1} and Proposition \ref{th2} (i) (see also Remark \ref{n=1}),
\[\|w\|_{L^p(\Omega)}\leq \|v\|_{L^1(\Omega)}\leq C\quad \textrm{for each}\ p<\frac{n}{n-2s}.\]
Since $\widetilde f(u_\lambda)^{2\gamma}$ is a subsolution of \eqref{hv2} ---by \eqref{defv}---,
it follows that
\[0\leq \widetilde f(u_\lambda)^{2\gamma}\leq w.\]
Therefore, $\|f(u_\lambda)\|_{L^p}\leq C$ for all $p<2\gamma\,\frac{n}{n-2s}$, where $C$ is a constant that does not depend on $\lambda$.
This can be done for any $\gamma\in(1,1+\sqrt{\tau})$, and thus we find
\begin{equation}\label{f(u)tau-}
\|f(u_\lambda)\|_{L^p}\leq C\ \ \textrm{for each}\ \ p<\frac{2n(1+\sqrt{\tau})}{n-2s}.
\end{equation}

Hence, using Proposition \ref{th2} (iii) and letting $\lambda\uparrow\lambda^*$ it follows that
\[u^*\in L^\infty(\Omega)\quad \mbox{whenever}\quad n<6s+4s\sqrt{\tau}.\]
Hence, the extremal solution is bounded whenever $n<10s$.

Assume now $\tau<1$.
In this case, Lemma \ref{lemapotencia} ensures that for each $\epsilon\in(0,1-\tau)$ there exist a constant $C$ such that
\begin{equation}\label{m}
f(t)\leq C(1+t)^{m},\qquad m=\frac{1}{1-(\tau+\epsilon)}.
\end{equation}
Then, by \eqref{f(u)tau-} we have that $\|f(u_\lambda)\|_{L^p}\leq C$ for each $p<p_0:=\frac{2n(1+\sqrt{\tau})}{n-2s}$.

Next we show that if $n<10s$ by a bootstrap argument we obtain $u^*\in L^\infty(\Omega)$.
Indeed, by Proposition \ref{th2} (ii) and \eqref{m} we have
\[f(u^*)\in L^p\quad \Longleftrightarrow\quad (-\Delta)^s u^*\in L^p\quad \Longrightarrow\quad u^*\in L^q
\quad \Longrightarrow\quad f(u^*)\in L^{q/m},\]
where $q=\frac{np}{n-2s p}$.
Now, we define recursively
\[p_{k+1}:=\frac{np_k}{m(n-2s p_k)},\qquad p_0=\frac{2n(1+\sqrt{\tau})}{n-2s}.\]
Now, since
\[p_{k+1}-p_k=\frac{p_k}{n-2s p_k}\left(2s p_k-\frac{m-1}{m}n\right),\]
then the bootstrap argument yields $u^*\in L^\infty(\Omega)$ in a finite number of steps provided that $(m-1)n/m<2s p_0$.
This condition is equivalent to $n<2s+4s\frac{1+\sqrt{\tau}}{\tau+\epsilon}$, which is satisfied for $\epsilon$ small enough whenever $n\leq10s$, since $\frac{1+\sqrt{\tau}}{\tau}>2$ for $\tau<1$. Thus, the result is proved.
\end{proof}

Before proving Theorem \ref{S} (i), we need the following lemma, proved by Nedev in \cite{N}.

\begin{lem}[\cite{N}]\label{lemah} Let $f$ be a convex function satisfying \eqref{condicions}, and let
\begin{equation}\label{defg}
g(t)=\int_0^t f'(\tau)^2d\tau.\end{equation}
Then,
\[\lim_{t\rightarrow +\infty}\frac{f'(t)\widetilde f(t)^2-\widetilde f(t)g(t)}{f(t)f'(t)}=+\infty,\]
where $\widetilde f(t)=f(t)-f(0)$.
\end{lem}

As said above, this lemma is proved in \cite{N}.
More precisely, see equation (6) in the proof of Theorem 1 in \cite{N} and recall that $\widetilde f/f\rightarrow1$ at infinity.

We can now give the

\begin{proof}[Proof of Theorem \ref{S} (i)]
Let $g$ be given by \eqref{defg}.
Using Lemma \ref{po} with $\gamma=1$ and integrating by parts, we find
\begin{equation}\label{ineqnedevf(u)}
\begin{split}
\left\|f(u_\lambda)\right\|_{\accentset{\circ}{H}^s}^2 &= \frac{c_{n,s}}{2}\int_{\R^n}\int_{\R^n}\frac{\left(f(u_\lambda(x))-f(u_\lambda(y))\right)^2}{|x-y|^{n+2s}}dxdy\\
&\leq \frac{c_{n,s}}{2}\int_{\R^n}\int_{\R^n}\frac{\left(g(u_\lambda(x))-g(u_\lambda(y))\right)\left(u_\lambda(x)-u_\lambda(y)\right)}{|x-y|^{n+2s}}dxdy\\
&= \int_{\R^n}(-\Delta)^{s/2}g(u_\lambda)(-\Delta)^{s/2}u_\lambda dx\\
&= \int_{\R^n}g(u_\lambda)(-\Delta)^s u_\lambda dx\\
&=\int_\Omega f(u_\lambda)g(u_\lambda).
\end{split}
\end{equation}

The stability condition (\ref{semistable}) applied with $\eta=\widetilde f(u_\lambda)$ yields
\[\int_\Omega f'(u_\lambda)\widetilde f(u_\lambda)^2 \leq \|\widetilde f(u_\lambda)\|_{\accentset{\circ}{H}^s}^2,\]
which combined with \eqref{ineqnedevf(u)} gives
\begin{equation}\label{4nedev}
\int_\Omega f'(u_\lambda)\widetilde f(u_\lambda)^2\leq \int_\Omega f(u_\lambda)g(u_\lambda).\end{equation}
This inequality can be written as
\[\int_\Omega\left\{f'(u_\lambda)\widetilde f(u_\lambda)^2-\widetilde f(u_\lambda)g(u_\lambda)\right\}\leq f(0)\int_\Omega g(u_\lambda).\]
In addition, since $f$ is convex we have
\[g(t)=\int_0^t f'(s)^2ds\leq f'(t)\int_0^tf'(s)ds\leq f'(t)f(t),\]
and thus,
\[\int_\Omega\left\{f'(u_\lambda)\widetilde f(u_\lambda)^2-\widetilde f(u_\lambda)g(u_\lambda)\right\}\leq f(0)\int_\Omega f'(u_\lambda)f(u_\lambda).\]
Hence, by Lemma \ref{lemah} we obtain
\begin{equation}\label{h1}
\int_\Omega f(u_\lambda)f'(u_\lambda)\leq C.
\end{equation}

Now, on the one hand we have that
\[f(a)-f(b)\leq f'(a)(a-b),\]
since $f$ is increasing and convex.
This yields, as in \eqref{defv},
\[(-\Delta)^s\widetilde f(u_\lambda)\leq f'(u_\lambda)(-\Delta)^su_\lambda=f'(u_\lambda)f(u_\lambda):=v(x).\]

On the other hand, let $w$ the solution of the problem
\begin{equation}\label{hv}
\left\{ \begin{array}{rcll} (-\Delta)^s w &=&v&\textrm{in }\Omega \\
w&=&0&\textrm{on }\partial\Omega.
\end{array}\right.\end{equation}
By \eqref{h1} and Proposition \ref{th2} (i) (see also Remark \ref{n=1}),
\[\|w\|_{L^p(\Omega)}\leq \|v\|_{L^1(\Omega)}\leq C\textrm{ for each }p<\frac{n}{n-2s}.\]
Since $\widetilde f(u_\lambda)$ is a subsolution of \eqref{hv}, then $0\leq \widetilde f(u_\lambda)\leq w$.
Therefore,
\[\|f(u^*)\|_{L^p(\Omega)}\leq C\ \ \textrm{for each}\ \ p<\frac{n}{n-2s},\]
and using Proposition \ref{th2} (iii), we find
\[u^*\in L^\infty(\Omega)\ \ \textrm{whenever}\ \ n<4s,\]
as desired.
\end{proof}

\section{Boundary estimates: the moving planes method}
\label{sec-bdry}

In this section we prove Proposition \ref{bdyestimates}.
This will be done with the celebrated moving planes method \cite{GNN}, as in the classical boundary estimates for the Laplacian of de Figueiredo-Lions-Nussbaum \cite{FLN}.

The moving planes method has been applied to problems involving the fractional Laplacian by different authors; see for example \cite{CLO,BMW,FW}.
However, some of these results use the specific properties of the fractional Laplacian ---such as the extension problem of Caffarelli-Silvestre \cite{CS-ext}, or the Riesz potential expression for $(-\Delta)^{-s}$---, and it is not clear how to apply the method to more general integro-differential operators.
Here, we follow a different approach that allows more general nonlocal operators.

The main tool in the proof is the following maximum principle in small domains.

Recently, Jarohs and Weth \cite{JW} obtained a parabolic version of the maximum principle in small domains for the fractional Laplacian; see Proposition 2.4 in \cite{JW}.
The proof of their result is essentially the same that we present in this section.
Still, we think that it may be of interest to write here the proof for integro-differential operators with decreasing kernels.

\begin{lem}\label{mpsd}
Let $\Omega\subset\R^n$ be a domain satisfying $\Omega\subset \R^n_+=\{x_1>0\}$.
Let $K$ be a nonnegative function in $\R^n$, radially symmetric and decreasing, and satisfying
\[K(z)\geq c|z|^{-n-\nu}\quad\textrm{for all}\quad z\in B_1\]
for some positive constants $c$ and $\nu$, and let
\[L_K u(x)=\int_{\R^n}\bigl(u(y)-u(x)\bigr)K(x-y)dy.\]

Let $V\in L^\infty(\Omega)$ be any bounded function, and $w\in H^s(\R^n)$ be a bounded function satisfying
\begin{equation}\label{pbmpsd}
\left\{ \begin{array}{rcll} L_K w &=&V(x)w&\textrm{in }\Omega \\
w&\geq&0&\textrm{in }\R^n_+\setminus\Omega\\
w(x)&\geq&-w(x^*)&\textrm{in }\R^n_+,\end{array}\right.\end{equation}
where $x^*$ is the symmetric to $x$ with respect to the hyperplane $\{x_1=0\}$.
Then, there exists a positive constant $C_0$ such that if
\begin{equation}\label{cond-mpsd}
\left(1+\|V^-\|_{L^\infty(\Omega)}\right)|\Omega|^{\frac{\nu}{n}}\leq C_0,
\end{equation}
then $w\geq0$ in $\Omega$.
\end{lem}

\begin{rem} When $L_K$ is the fractional Laplacian $(-\Delta)^s$, then the condition \eqref{cond-mpsd} can be replaced by $\|V^-\|_{L^\infty}|\Omega|^{\frac{2s}{n}}\leq C_0$.
\end{rem}

\begin{proof}[Proof of Lemma \ref{mpsd}]
The identity $L_K w =V(x)w$ in $\Omega$ written in weak form is
\begin{equation}\label{mpsd1}
(\varphi,w)_K:=\int\int_{\R^{2n}\setminus(\R^n\setminus\Omega)^2}{(\varphi(x)-\varphi(y))(w(x)-w(y))}K(x-y)dx\,dy=\int_\Omega Vw\varphi
\end{equation}
for all $\varphi$ such that $\varphi\equiv0$ in $\R^n\setminus\Omega$ and $\int_{\R^n}\bigl(\varphi(x)-\varphi(y)\bigr)^2K(x-y)dx\,dy<\infty$.
Note that the left hand side of \eqref{mpsd1} can be written as
\[\begin{split}
(\varphi,w)_K=&\int_\Omega\int_\Omega {(\varphi(x)-\varphi(y))(w(x)-w(y))}K(x-y)dx\,dy\\
&+2\int_{\Omega}\int_{\R^n_+\setminus\Omega} {\varphi(x)(w(x)-w(y))}K(x-y)dx\,dy\\
&+2\int_{\Omega}\int_{\R^n_+} {\varphi(x)(w(x)-w(y^*))}K(x-y^*)dx\,dy,\end{split}\]
where $y^*$ denotes the symmetric of $y$ with respect to the hyperplane $\{x_1=0\}$.

Choose $\varphi=-w^-\chi_\Omega$, where $w^-$ is the negative part of $w$, i.e., $w=w^+-w^-$.
Then, we claim that
\begin{equation}\label{claim}
\int\int_{\R^{2n}\setminus(\R^n\setminus\Omega)^2} {(w^-(x)\chi_\Omega(x)-w^-(y)\chi_\Omega(y))^2}K(x-y)dx\,dy\leq (-w^-\chi_\Omega,w)_K.
\end{equation}
Indeed, first, we have
\[\begin{split}
(-w^-\chi_\Omega,w)_K&= \int_\Omega\int_\Omega \{(w^-(x)\hspace{-1mm}-\hspace{-1mm}w^-(y))^2\hspace{-1mm}+\hspace{-1mm}w^-(x)w^+(y)\hspace{-1mm}+\hspace{-1mm}w^+(x)w^-(y)\}K(x\hspace{-1mm}-\hspace{-1mm}y)dxdy+\\
&+2\int_{\Omega}\int_{\R^n_+\setminus\Omega} \{w^-(x)(w^-(x)-w^-(y))+w^-(x)w^+(y)\}K(x-y)dx\,dy\\
&+2\int_{\Omega}\int_{\R^n_+} \{w^-(x)(w^-(x)-w^-(y^*))+w^-(x)w^+(y^*)\}K(x-y^*)dx\,dy,
\end{split}\]
where we have used that $w^+(x)w^-(x)=0$ for all $x\in \R^n$.

Thus, rearranging terms and using that $w^-\equiv0$ in $\R^n_+\setminus\Omega$,
\[\begin{split}
(-w^-\chi_\Omega,w)_K=& \int\int_{\R^{2n}\setminus(\R^n\setminus\Omega)^2} {(w^-(x)\chi_\Omega(x)-w^-(y)\chi_\Omega(y))^2}K(x-y)dx\,dy\\
&+\int_\Omega\int_\Omega {2w^-(x)w^+(y)}K(x-y)dx\,dy+\\
&+2\int_{\Omega}\int_{\R^n_+\setminus\Omega} \{w^-(x)w^+(y)-w^-(x)w^-(y)\}K(x-y)dx\,dy\\
&+2\int_{\Omega}\int_{\R^n_+} \{w^-(x)w^+(y^*)-w^-(x)w^-(y^*)\}K(x-y^*)dx\,dy\\
\geq&\int\int_{\R^{2n}\setminus(\R^n\setminus\Omega)^2} {(w^-(x)\chi_\Omega(x)-w^-(y)\chi_\Omega(y))^2}K(x-y)dx\,dy+\\
&+2\int_\Omega\int_{\R^n_+} {w^-(x)w^+(y)}K(x-y)dx\,dy+\\
&+2\int_{\Omega}\int_{\R^n_+} {-w^-(x)w^-(y^*)}K(x-y^*)dx\,dy.
\end{split}\]
We next use that, since $K$ is radially symmetric and decreasing, $K(x-y^*)\leq K(x-y)$ for all $x$ and $y$ in $\R^n_+$.
We deduce
\[\begin{split}(-w^-\chi_\Omega,w)_K\geq&
\int\int_{\R^{2n}\setminus(\R^n\setminus\Omega)^2} {(w^-(x)\chi_\Omega(x)-w^-(y)\chi_\Omega(y))^2}K(x-y)dx\,dy+\\
&+2\int_{\Omega}\int_{\R^n_+} {w^-(x)w^+(y)-w^-(x)w^-(y^*)}K(x-y)dx\,dy,
\end{split}\]
and since $w^-(y^*)\leq w^+(y)$ for all $y$ in $\R^n_+$ by assumption, we obtain \eqref{claim}.

Now, on the one hand note that from \eqref{claim} we find
\[\int_{\Omega}\int_{\Omega} {(w^-(x)-w^-(y))^2}K(x-y)dx\,dy\leq (-w^-\chi_\Omega,w)_K.\]
Moreover, since $K(z)\geq c|z|^{-n-\nu}\chi_{B_1}(z)$, then
\[\begin{split}
\|w^-\|^2_{\accentset{\circ}{H}^{\nu/2}(\Omega)}&:=\frac{c_{n,s}}{2}\int_{\Omega}\int_{\Omega} \frac{(w^-(x)-w^-(y))^2}{|x-y|^{-n-\nu}}dx\,dy\\
&\leq C\|w^-\|_{L^2(\Omega)}+C\int_{\Omega}\int_{\Omega} {\bigl(w^-(x)-w^-(y)\bigr)^2}K(x-y)dx\,dy,
\end{split}\]
and therefore
\begin{equation}\label{mpsd3}
\|w^-\|^2_{\accentset{\circ}{H}^{\nu/2}(\Omega)}\leq C_1\|w^-\|_{L^2(\Omega)}+C_1(-w^-\chi_\Omega,w)_K.
\end{equation}

On the other hand, it is clear that
\begin{equation}\label{mpsd2}
\int_\Omega Vww^-=\int_\Omega V(w^-)^2\leq \|V^-\|_{L^\infty(\Omega)}\|w^-\|_{L^2(\Omega)}.\end{equation}

Thus, it follows from \eqref{mpsd1}, \eqref{mpsd3}, and \eqref{mpsd2} that
\[\|w^-\|_{\accentset{\circ}{H}^{\nu/2}(\Omega)}^2\leq  C_1\left(1+\|V^-\|_{L^\infty}\right)\|w^-\|_{L^2(\Omega)}.\]
Finally, by the H\"older and the fractional Sobolev inequalities, we have
\[\|w^-\|_{L^2(\Omega)}^2\leq |\Omega|^{\frac{\nu}{n}}\|w^-\|_{L^q(\Omega)}^2\leq C_2|\Omega|^{\frac{\nu}{n}}\|w^-\|_{\accentset{\circ}{H}^{\nu/2}(\Omega)}^2,\]
where $q=\frac{2n}{n-\nu}$.
Thus, taking $C_0$ such that $C_0<(C_1C_2)^{-1}$ the lemma follows.
\end{proof}

Now, once we have the nonlocal version of the maximum principle in small domains, the moving planes method can be applied exactly as in the classical case.

\begin{proof}[Proof of Proposition \ref{bdyestimates}]
Replacing the classical maximum principle in small domains by Lemma \ref{mpsd}, we can apply the moving planes method to deduce
$\|u\|_{L^\infty(\Omega_\delta)}\leq C\|u\|_{L^1(\Omega)}$ for some constants $C$ and $\delta>0$ that depend only on $\Omega$, as in de Figueiredo-Lions-Nussbaum \cite{FLN}; see also \cite{B}.

Let us recall this argument.
Assume first that all curvatures of $\partial\Omega$ are positive.
Let $\nu(y)$ be the unit outward normal to $\Omega$ at $y$.
Then, there exist positive constants $s_0$ and $\alpha$
depending only on the convex domain $\Omega$ such that, for every $y\in\partial\Omega$
and every $e\in \R^n$ with $|e| = 1$ and $e\cdot \nu(y)\geq\alpha$, $u(y-se)$ is nondecreasing in
$s\in[0, s_0]$.
This fact follows from the moving planes method applied to planes close
to those tangent to $\Omega$ at $\partial\Omega$.
By the convexity of $\Omega$, the reflected caps will be
contained in $\Omega$.
The previous monotonicity fact leads to the existence of a set $I_x$, for each $x\in\Omega_\delta$, and a constant $\gamma>0$ that depend only on $\Omega$, such that
\[|I_x|\geq \gamma,\qquad u(x)\leq u(y)\quad \textrm{for all}\quad y\in I_x.\]
The set $I_x$ is a truncated open cone with vertex at $x$.

As mentioned in page 45 of de Figuereido-Lions-Nussbaum \cite{FLN},
the same can also be proved for general convex domains with a little more of care.
\end{proof}

\begin{rem}
When $\Omega=B_1$, Proposition \ref{bdyestimates} follows from the results in \cite{BMW}, where Birkner, L\'opez-Mimbela, and Wakolbinger used the moving planes method to show that any nonnegative bounded solution of
\begin{equation}\label{semilinear}
\left\{\begin{array}{rcll}
(-\Delta)^s u &=&f(u)&\quad \textrm{in}\ B_1\\
u&=&0&\quad \textrm{in}\ \R^n\setminus B_1
\end{array}\right.
\end{equation}
is radially symmetric and decreasing.

When $u$ is a bounded semistable solution of \eqref{semilinear}, there is an alternative way to show that $u$ is radially symmetric.
This alternative proof applies to all solutions (not necessarily positive), but does not give monotonicity.
Indeed, one can easily show that, for any $i\neq j$, the function $w=x_iu_{x_j}-x_j u_{x_i}$ is a solution of the linearized problem
\begin{equation}
\left\{\begin{array}{rcll}
(-\Delta)^s w &=&f'(u)w&\quad \textrm{in}\ B_1\\
w&=&0&\quad \textrm{in}\ \R^n\setminus B_1.
\end{array}\right.
\end{equation}
Then, since $\lambda_1\left((-\Delta)^s-f'(u);B_1\right)\geq0$ by assumption, it follows that either $w\equiv0$ or $\lambda_1=0$ and $w$ is a multiple of the first eigenfunction, which is positive ---see the proof of Proposition 9 in \cite[Appendix A]{SV2}.
But since $w$ is a tangential derivative then it can not have constant sign along a circumference $\{|x|=r\}$, $r\in(0,1)$, and thus it has to be $w\equiv0$.
Therefore, all the tangential derivatives $\partial_t u=x_iu_{x_j}-x_j u_{x_i}$ equal zero, and thus $u$ is radially symmetric.
\end{rem}

\section{$H^s$ regularity of the extremal solution in convex domains}
\label{sec-Hs}

In this section we prove Theorem \ref{S} (iii).
A key tool in this proof is the Pohozaev identity for the fractional Laplacian, recently obtained by the authors in \cite{RS}.
This identity allows us to compare the interior $H^s$ norm of the extremal solution $u^*$ with a boundary term involving $u^*/\delta^s$, where $\delta$ is the distance to $\partial\Omega$.
Then, this boundary term can be bounded by using the results of the previous section by the $L^1$ norm of $u^*$, which is finite.

We first prove the boundedness of $u^*/\delta^s$ near the boundary.

\begin{lem}\label{bdyestimates-quotient}
Let $\Omega$ be a convex domain, $u$ be a bounded solution of \eqref{Omega},
and $\delta(x)=\textrm{dist}(x,\partial\Omega)$.
Assume that
\[\|u\|_{L^1(\Omega)}\leq c_1\]
for some $c_1>0$.
Then, there exists constants $\delta>0$, $c_2$, and $C$ such that
\[\|u/\delta^s\|_{L^{\infty}(\Omega_\delta)}\leq C\left(c_2+\|f\|_{L^\infty([0,c_2])}\right),\]
where $\Omega_\delta=\{x\in\Omega\,:\, {\rm dist}(x,\partial\Omega)<\delta\}$.
Moreover, the constants $\delta$, $c_2$, and $C$ depend only on $\Omega$ and $c_1$.
\end{lem}

\begin{proof}
The result can be deduced from the boundary regularity results in \cite{RS-Dir} and Proposition \ref{bdyestimates}, as follows.

Let $\delta>0$ be given by Proposition \ref{bdyestimates}, and let $\eta$ be a smooth cutoff function satisfying $\eta\equiv0$ in $\Omega\setminus\Omega_{2\delta/3}$ and $\eta\equiv1$ in $\Omega_{\delta/3}$.
Then, $u\eta\in L^\infty(\Omega)$ and $u\eta\equiv0$ in $\R^n\setminus\Omega$.
Moreover, we claim that
\begin{equation}\label{eqt}
(-\Delta)^s(u\eta)=f(u)\chi_{\Omega_{\delta/4}}+g\qquad \textrm{in}\ \Omega
\end{equation}
for some function $g\in  L^\infty(\Omega)$, with the estimate
\begin{equation}\label{eqt2}
\|g\|_{L^\infty(\Omega)}\leq C\left(\|u\|_{C^{1+s}(\Omega_{4\delta/5}\setminus\Omega_{\delta/5})}+\|u\|_{L^1(\Omega)}\right).
\end{equation}

To prove that \eqref{eqt} holds pointwise we argue separately in $\Omega_{\delta/4}$, in $\Omega_{3\delta/4}\setminus\Omega_{\delta/4}$, and in $\Omega\setminus\Omega_{3\delta/4}$, as follows:
\begin{itemize}
\item In $\Omega_{\delta/4}$, $g=(-\Delta)^s (u\eta)-(-\Delta)^s u$.
Since $u\eta-u$ vanishes in $\Omega_{\delta/3}$ and also outside $\Omega$, $g$ is bounded and satisfies \eqref{eqt2}.

\item In $\Omega_{3\delta/4}\setminus\Omega_{\delta/4}$, $g=(-\Delta)^s(u\eta)$.
Then, using
\[\|(-\Delta)^s(u\eta)\|_{L^\infty(\Omega_{3\delta/4}\setminus\Omega_{\delta/4})}\leq C\left(\|u\eta\|_{C^{1+s}(\Omega_{4\delta/5}\setminus\Omega_{\delta/5})}+\|u\eta\|_{L^1(\R^n)}\right)\]
and that $\eta$ is smooth, we find that $g$ is bounded and satisfies \eqref{eqt2}.

\item In $\Omega\setminus\Omega_{3\delta/4}$, $g=(-\Delta)^s(u\eta)$.
Since $u\eta$ vanishes in $\Omega\setminus\Omega_{2\delta/3}$, $g$ is bounded and satisfies \eqref{eqt2}.
\end{itemize}

Now, since $u$ is a solution of \eqref{Omega}, by classical interior estimates we have
\begin{equation}\label{eqt3}
\|u\|_{C^{1+s}(\Omega_{4\delta/5}\setminus\Omega_{\delta/5})}\leq C\left(\|u\|_{L^\infty(\Omega_\delta)}+\|u\|_{L^1(\Omega)}\right);
\end{equation}
see for instance \cite{RS-Dir}.
Hence, by \eqref{eqt} and Theorem 1.2 in \cite{RS-Dir}, $u\eta/\delta^s\in C^{\alpha}(\overline\Omega)$ for some $\alpha>0$ and
\[\|u\eta/\delta^s\|_{C^{\alpha}(\overline\Omega)}\leq C\|f(u)\chi_{\Omega_{\delta/4}}+g\|_{L^\infty(\Omega)}.\]
Thus,
\[\begin{split}
\|u/\delta^s\|_{L^\infty(\Omega_{\delta/3})}&
\leq \|u\eta/\delta^s\|_{C^{\alpha}(\overline\Omega)}\leq C\left(\|g\|_{L^\infty(\Omega)}+\|f(u)\|_{L^{\infty}(\Omega_{\delta/4})}\right)\\
&\leq C\left(\|u\|_{L^1(\Omega)}+\|u\|_{L^\infty(\Omega_\delta)}+\|f(u)\|_{L^\infty(\Omega_{\delta/4})}\right).
\end{split}\]
In the last inequality we have used \eqref{eqt2} and \eqref{eqt3}.
Then, the result follows from Proposition \ref{bdyestimates}.
\end{proof}

We can now give the

\begin{proof}[Proof of Theorem \ref{S} (iii)] Recall that $u_\lambda$ minimizes the energy $\mathcal E$ in the set $\{u\in H^s(\R^n)\,:\,0\leq u\leq u_\lambda\}$ (see Step 4 in the proof of Proposition \ref{existence} in Section \ref{sec-exist}).
Hence,
\begin{equation}\label{p1}
\|u_\lambda\|_{\accentset{\circ}{H}^s}^2-\int_\Omega \lambda F(u_\lambda)=\mathcal E(u_\lambda)\leq \mathcal E(0)=0.\end{equation}
Now, the Pohozaev identity for the fractional Laplacian can be written as
\begin{equation}\label{p2}s\|u_\lambda\|_{\accentset{\circ}{H}^s}^2-n\mathcal E(u_\lambda)=\frac{\Gamma(1+s)^2}{2}\int_{\partial\Omega}\left(\frac{u_\lambda}{\delta^s}\right)^2(x\cdot \nu)d\sigma,\end{equation}
see \cite[page 2]{RS}.
Therefore, it follows from \eqref{p1} and \eqref{p2} that
\[\|u_\lambda\|_{\accentset{\circ}{H}^s}^2\leq \frac{\Gamma(1+s)^2}{2s}\int_{\partial\Omega}\left(\frac{u_\lambda}{\delta^s}\right)^2(x\cdot \nu)d\sigma.\]

Now, by Proposition \ref{bdyestimates-quotient}, we have that
\[\int_{\partial\Omega}\left(\frac{u_\lambda}{\delta^s}\right)^2(x\cdot \nu)d\sigma\leq C\]
for some constant $C$ that depends only on $\Omega$ and $\|u_\lambda\|_{L^1(\Omega)}$.
Thus, $\|u_\lambda\|_{\accentset{\circ}{H}^s}\leq  C$, and since $u^*\in L^1(\Omega)$, letting $\lambda\uparrow \lambda^*$ we find
\[\|u^*\|_{\accentset{\circ}{H}^s}<\infty,\]
as desired.
\end{proof}

\section{$L^p$ and $C^\beta$ estimates for the linear Dirichlet problem}
\label{sec3}

The aim of this section is to prove Propositions \ref{th2} and \ref{prop:u-is-Cbeta}.
We prove first Proposition \ref{th2}.

\begin{proof}[Proof of Proposition \ref{th2}]
(i) It is clear that we can assume $\|g\|_{L^1(\Omega)}=1$.

Consider the solution $v$ of
\[(-\Delta)^s v=|g|\ \ {\rm in}\ \R^n\]
given by the Riesz potential $v=(-\Delta)^{-s}|g|$.
Here, $g$ is extended by 0 outside $\Omega$.

Since $v\geq0$ in $\R^n\setminus\Omega$, by the maximum principle we have that $|u|\leq v$ in $\Omega$.
Then, it follows from Theorem \ref{th1} that
\[\|u\|_{L^q_{\textrm{weak}}(\Omega)}\leq C,\quad\textrm{where}\quad q=\frac{n}{n-2s},\]
and hence we find that
\[\|u\|_{L^r(\Omega)}\leq C\quad\textrm{for all}\quad r<\frac{n}{n-2s}\]
for some constant that depends only on $n$, $s$, and $|\Omega|$.

(ii) The proof is analogous to the one of part (i).
In this case, the constant does not depend on the domain $\Omega$.

(iii) As before, we assume $\|g\|_{L^p(\Omega)}=1$.
Write $u=\tilde v+w$, where $\tilde v$ and $w$ are given by
\begin{equation}\label{v}
\tilde v=(-\Delta)^{-s} g\ \ {\rm in}\ \R^n,
\end{equation}
and
\begin{equation}\label{w}
\left\{ \begin{array}{rcll} (-\Delta)^s w &=&0&\textrm{in }\Omega \\
w&=&\tilde v&\textrm{in }\mathbb R^n\backslash\Omega.\end{array}\right.
\end{equation}
Then, from \eqref{v} and Theorem \ref{th1} we deduce that
\begin{equation}\label{frt}
[\tilde v]_{C^\alpha(\R^n)}\leq C,\quad\textrm{where}\quad \alpha=2s-\frac{n}{p}.
\end{equation}
Moreover, since the domain $\Omega$ is bounded, then $g$ has compact support and hence $\tilde v$ decays at infinity.
Thus, we find
\begin{equation}\label{ineqv}
\|\tilde v\|_{C^\alpha(\R^n)}\leq C
\end{equation}
for some constant $C$ that depends only on $n$, $s$, $p$, and $\Omega$.

Now, we apply Proposition \ref{prop:u-is-Cbeta} to equation \eqref{w}.
We find
\begin{equation}\label{ineqw}
\|w\|_{C^\beta(\R^n)}\leq C\|\tilde v\|_{C^\alpha(\R^n)},\end{equation}
where $\beta=\min\{\alpha,s\}$.
Thus, combining \eqref{ineqv}, and \eqref{ineqw} the result follows.
\end{proof}

Note that we have only used Proposition \ref{prop:u-is-Cbeta} to obtain the $C^\beta$ estimate in part (iii).
If one only needs an $L^\infty$ estimate instead of the $C^\beta$ one, Proposition \ref{prop:u-is-Cbeta} is not needed, since the $L^\infty$ bound follows from the maximum principle.

As said in the introduction, the $L^p$ to $W^{2s,p}$ estimates for the fractional Laplace equation, in which $-\Delta$ is replaced by the fractional Laplacian $(-\Delta)^s$, are not true for all $p$, even when $\Omega=\R^n$.
This is illustrated in the following two remarks.

Recall the definition of the fractional Sobolev space $W^{\sigma,p}(\Omega)$ which, for $\sigma\in(0,1)$, consists of all functions $u\in L^p(\Omega)$ such that
\[\|u\|_{W^{\sigma,p}(\Omega)}=\|u\|_{L^p(\Omega)}+\left(\int_{\Omega}\int_{\Omega}\frac{|u(x)-u(y)|^p}{|x-y|^{n+p\sigma}}\,dx\,dy\right)^{\frac1p}\]
is finite; see for example \cite{DPV} for more information on these spaces.

\begin{rem}\label{rem1} Let $s\in(0,1)$.
Assume that $u$ and $g$ belong to $L^p(\R^n)$, with $1<p<\infty$, and that
\[(-\Delta)^su=g\ \ \textrm{in}\ \R^n.\]
\begin{itemize}
\item[(i)] If $p\geq2$, then $u\in W^{2s,p}(\R^n)$.
\item[(ii)] If $p<2$ and $2s\neq1$ then $u$ may not belong to $W^{2s,p}(\R^n)$.
Instead, $u\in B^{2s}_{p,2}(\R^n)$, where $B^\sigma_{p,q}$ is the Besov space of order $\sigma$ and parameters $p$ and $q$.
\end{itemize}
For more details see the books of Stein \cite{Stein} and Triebel \cite{T2}.
\end{rem}

By the preceding remark we see that the $L^p$ to $W^{2s,p}$ estimate does not hold in $\R^n$ whenever $p<2$ and $s\neq \frac12$.
The following remark shows that in bounded domains $\Omega$ this estimate do not hold even for $p\geq2$.

\begin{rem}\label{rem11}
Let us consider the solution of $(-\Delta)^su=g$ in $\Omega$, $u\equiv0$ in $\R^n\setminus\Omega$.
When $\Omega=B_1$ and $g\equiv1$, the solution to this problem is
\[u_0(x)=\left(1-|x|^2\right)^s\chi_{B_1}(x);\]
see \cite{G}.
For $p$ large enough one can see that $u_0$ does not belong to $W^{2s,p}(B_1)$, while $g\equiv1$ belongs to $L^p(B_1)$ for all $p$.
For example, when $s=\frac12$ by computing $|\nabla u_0|$ we see that $u_0$ does not belong to $W^{1,p}(B_1)$ for $p\geq2$.
\end{rem}

We next prove Proposition \ref{prop:u-is-Cbeta}.
For it, we will proceed similarly to the $C^s$ estimates obtained in \cite[Section~2]{RS-Dir} for the Dirichlet problem for the fractional Laplacian with $L^\infty$ data.

The first step is the following:

\begin{lem}\label{bound-u} Let $\Omega$ be a bounded domain satisfying the exterior ball condition, $s\in(0,1)$,
$h$ be a $C^\alpha(\R^n\setminus\Omega)$ function for some $\alpha>0$, and $u$ be the solution of \eqref{g}.
Then
\[|u(x)-u(x_0)|\leq C\|h\|_{C^\alpha(\R^n\setminus\Omega)}\delta(x)^\beta\ \ {\rm in}\ \Omega,\]
where $x_0$ is the nearest point to $x$ on $\partial\Omega$, $\beta=\min\{s,\alpha\}$, and $\delta(x)={\rm dist}(x,\partial\Omega)$.
The constant $C$ depends only on $n$, $s$, and $\alpha$.
\end{lem}

Lemma \ref{bound-u} will be proved using the following supersolution.
Next lemma (and its proof) is very similar to Lemma 2.6 in \cite{RS-Dir}.

\begin{lem}\label{supersolution} Let $s\in(0,1)$.
Then, there exist constants $\epsilon$, $c_1$, and $C_2$, and a continuous radial function $\varphi$ satisfying
\begin{equation}\label{supers}
\begin{cases}
(-\Delta)^s \varphi \ge 0 &\mbox{in }B_2\setminus B_1\\
\varphi \equiv 0 \quad &\mbox{in }B_1 \\
c_1(|x|-1)^s\le\varphi \le C_2(|x|-1)^s &\mbox{in }\R^n\setminus B_1.
\end{cases}
\end{equation}
The constants $c_1$ and $C_2$ depend only on $n$, $s$, and $\beta$.
\end{lem}

\begin{proof} We follow the proof of Lemma 2.6 in \cite{RS-Dir}.
Consider the function
\[u_0(x)=(1-|x|^2)^s_+.\]
It is a classical result (see \cite{G}) that this function satisfies
\[(-\Delta)^su_0= \kappa_{n,s}\ \ {\rm in}\ B_1\]
for some positive constant $\kappa_{n,s}$.

Thus, the fractional Kelvin transform of $u_0$, that we denote by $u_0^*$, satisfies
\[(-\Delta)^s u_0^*(x)=|x|^{-2s-n}(-\Delta)^su_0\left(\frac{x}{|x|^2}\right)\geq c_0\ \ {\rm in}\ B_2\setminus B_1.\]
Recall that the Kelvin transform $u_0^*$ of $u_0$ is defined by
\[u_0^*(x)=|x|^{2s-n}u_0\left(\frac{x}{|x|^2}\right).\]
Then, it is clear that
\[a_1(|x|-1)^s\leq u_0^*(x)\leq A_2(|x|-1)^s\ \ {\rm in}\ B_{2}\setminus B_1,\]
while $u_0^*$ is bounded at infinity.

Let us consider now a smooth function $\eta$ satisfying $\eta\equiv0$ in $B_3$ and
\[A_1(|x|-1)^s\leq \eta\leq A_2(|x|-1)^s\ \ {\rm in}\ \R^n\setminus B_4.\]
Observe that $(-\Delta)^s\eta$ is bounded in $B_2$, since $\eta(x)(1+|x|)^{-n-2s}\in L^1$.
Then, the function
\[\varphi=Cu_0^*+\eta,\]
for some big constant $C>0$, satisfies
\[
\begin{cases}
(-\Delta)^s \varphi \ge 1 &\mbox{in }B_{2}\setminus B_1\\
\varphi \equiv 0 \quad &\mbox{in }B_1 \\
c_1(|x|-1)^s\le\varphi \le C_2(|x|-1)^s &\mbox{in }\R^n\setminus B_1.
\end{cases}
\]
Indeed, it is clear that $\varphi\equiv0$ in $B_1$.
Moreover, taking $C$ big enough it is clear that we have that $(-\Delta)^s \varphi \ge 1$.
In addition, the condition $c_1(|x|-1)^s\le\varphi \le C_2(|x|-1)^s$ is satisfied by construction.
Thus, $\varphi$ satisfies \eqref{supers}, and the proof is finished.
\end{proof}

Once we have constructed the supersolution, we can give the

\begin{proof}[Proof of Lemma \ref{bound-u}]
First, we can assume that $\|h\|_{C^\alpha(\R^n\setminus\Omega)}=1$.
Then, by the maximum principle we have that $\|u\|_{L^\infty(\R^n)}=\|h\|_{L^\infty(\R^n)}\leq1$.
We can also assume that $\alpha\leq s$, since
\[\|h\|_{C^s(\R^n)}\leq C\|h\|_{C^\alpha(\R^n\setminus\Omega)}\quad\textrm{whenever}\ s<\alpha.\]

Let $x_0\in \partial\Omega$ and $R>0$ be small enough.
Let $B_R$ be a ball of radius $R$, exterior to $\Omega$, and touching $\partial\Omega$ at $x_0$.
Let us see that $|u(x)-u(x_0)|$ is bounded by $CR^\beta$ in $\Omega\cap B_{2R}$.

By Lemma \ref{supersolution}, we find that there exist constants $c_1$ and $C_2$, and a radial continuous function $\varphi$ satisfying
\begin{equation}\label{supers}
\begin{cases}
(-\Delta)^s \varphi \ge 0 &\mbox{in }B_2\setminus B_1\\
\varphi \equiv 0 \quad &\mbox{in }B_1 \\
c_1(|x|-1)^s\le\varphi \le C_2(|x|-1)^s &\mbox{in }\R^n\setminus B_1.
\end{cases}
\end{equation}

\begin{figure}[h]
\begin{center}
\includegraphics[]{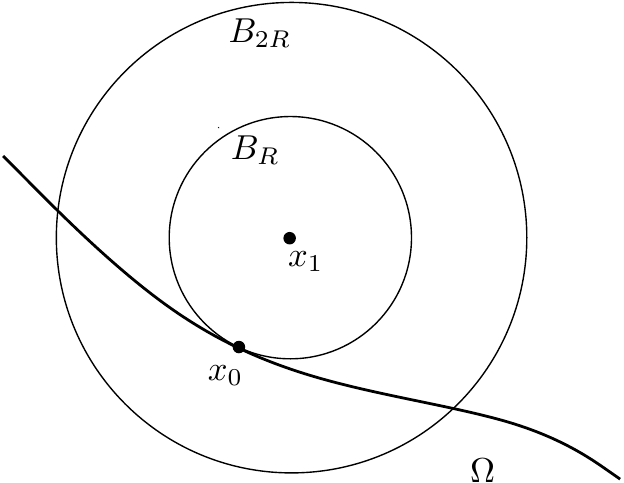}
\end{center}
\caption{\label{figura2} }
\end{figure}

Let $x_1$ be the center of the ball $B_R$.
Since $\|h\|_{C^\alpha(\R^n\setminus\Omega)}=1$, it is clear that the function
\[\varphi_R(x)=h(x_0)+3R^\alpha+C_3R^s\varphi\left(\frac{x-x_1}{R}\right),\]
with $C_3$ big enough, satisfies
\begin{equation}\label{supersR}
\begin{cases}
(-\Delta)^s \varphi_R \ge 0 &\mbox{in }B_{2R}\setminus B_R\\
\varphi_R \equiv h(x_0)+3R^\alpha \quad &\mbox{in }B_R \\
h(x_0)+|x-x_0|^\alpha\leq \varphi_R&\mbox{in }\R^n\setminus B_{2R}\\
\varphi_R \le h(x_0)+C_0R^\alpha &\mbox{in }B_{2R}\setminus B_R.
\end{cases}
\end{equation}
Here we have used that $\alpha\leq s$.

Then, since
\[(-\Delta)^s u\equiv0\leq (-\Delta)^s \varphi_R\quad \textrm{in}\ \Omega\cap B_{2R},\]
\[h\leq h(x_0)+3R^\alpha\equiv\varphi_R\quad \textrm{in}\ B_{2R}\setminus\Omega,\]
and
\[h(x)\leq h(x_0)+|x-x_0|^\alpha\leq\varphi_R\quad \textrm{in}\ \R^n\setminus B_{2R},\]
it follows from the comparison principle that
\[u\leq \varphi_R\ \ {\rm in}\ \Omega\cap B_{2R}.\]
Therefore, since $\varphi_R \le h(x_0)+C_0R^\alpha$ in $B_{2R}\setminus B_R$,
\begin{equation}\label{eqhatu}
u(x)- h(x_0)\leq C_0R^\alpha\ \ {\rm in}\ \Omega\cap B_{2R}.\end{equation}

Moreover, since this can be done for each $x_0$ on $\partial\Omega$, $h(x_0)=u(x_0)$, and we have $\|u\|_{L^\infty(\Omega)}\leq 1$, we find that
\begin{equation}\label{equ}
u(x)-u(x_0)\leq C\delta^\beta\ \ {\rm in}\ \Omega,
\end{equation}
where $x_0$ is the projection on $\partial\Omega$ of $x$.

Repeating the same argument with $u$ and $h$ replaced by $-u$ and $-h$, we obtain the same bound for $h(x_0)-u(x)$, and thus the lemma follows.
\end{proof}

The following result will be used to obtain $C^\beta$ estimates for $u$ inside $\Omega$.
For a proof of this lemma see for example Corollary 2.4 in \cite{RS-Dir}.

\begin{lem}[\cite{RS-Dir}]\label{int-est-brick2}
Let $s\in(0,1)$, and let $w$ be a solution of $(-\Delta)^s w = 0$ in $B_2$.
Then, for every $\gamma\in(0,2s)$
\[ \|w\|_{C^{\gamma}(\overline{B_{1/2}})} \le C\biggl(\|(1+|x|)^{-n-2s}w(x)\|_{L^1(\R^n)} + \|w\|_{L^\infty(B_2)}\biggr),\]
where the constant $C$ depends only on $n$, $s$, and $\gamma$.
\end{lem}

Now, we use Lemmas \ref{bound-u} and \ref{int-est-brick2} to obtain interior $C^\beta$ estimates for the solution of \eqref{g}.

\begin{lem}\label{Cbetabounds}
Let $\Omega$ be a bounded domain satisfying the exterior ball condition, $h\in C^\alpha(\R^n\setminus\Omega)$ for some $\alpha>0$, and $u$ be the solution of \eqref{g}.
Then, for all $x\in \Omega$ we have the following estimate in $B_R(x)=B_{\delta(x)/2}(x)$
\begin{equation}\label{seminorm-estimate-u}
\|u\|_{C^\beta(\overline{B_{R}(x)})}\le C\|h\|_{C^\alpha(\R^n\setminus\Omega)},
\end{equation}
where $\beta=\min\{\alpha,s\}$ and $C$ is a constant depending only on $\Omega$, $s$, and $\alpha$.
\end{lem}

\begin{proof}
Note that $B_R(x)\subset B_{2R}(x)\subset \Omega$.
Let $\tilde u(y)= u(x+Ry)-u(x)$.
We have that
\begin{equation}\label{Rs1}
(-\Delta)^s \tilde u(y) =  0\quad \mbox{in }  B_1\,.
\end{equation}
Moreover, using Lemma \ref{bound-u} we obtain
\begin{equation}\label{Rs2}
\|\tilde u\|_{L^\infty(B_1)}\le C\|h\|_{C^\alpha(\R^n\setminus\Omega)} R^\beta.
\end{equation}
Furthermore, observing that $|\tilde u(y)|\le C\|h\|_{C^\alpha(\R^n\setminus\Omega)} R^\beta(1+|y|^\beta)$ in all of $\R^n$, we find
\begin{equation}\label{Rs3}
\|(1+|y|)^{-n-2s}\tilde u(y)\|_{L^1(\R^n)}\le C \|h\|_{C^\alpha(\R^n\setminus\Omega)} R^\beta,
\end{equation}
with $C$ depending only on $\Omega$, $s$, and $\alpha$.

Now, using Lemma \ref{int-est-brick2} with $\gamma=\beta$, and taking into account \eqref{Rs1}, \eqref{Rs2}, and \eqref{Rs3}, we deduce
\[\|\tilde u\|_{C^{\beta}\left(\overline{B_{1/4}}\right)} \le C \|h\|_{C^\alpha(\R^n\setminus\Omega)}R^\beta,\]
where $C=C(\Omega,s,\beta)$.

Finally, we observe that
\[[u]_{C^\beta\left(\overline{B_{R/4}(x)}\right)}=R^{-\beta}[\tilde u]_{C^\beta\left(\overline{B_{1/4}}\right)}.\]
Hence, by an standard covering argument, we find the estimate \eqref{seminorm-estimate-u} for the $C^\beta$ norm of $u$ in $\overline {B_{R}(x)}$.
\end{proof}

Now, Proposition \ref{prop:u-is-Cbeta} follows immediately from Lemma \ref{Cbetabounds}, as in Proposition 1.1 in \cite{RS-Dir}.

\begin{proof}[Proof of Proposition \ref{prop:u-is-Cbeta}]
This proof is completely analogous to the proof of Proposition 1.1 in \cite{RS-Dir}.
One only have to replace the $s$ in that proof by $\beta$, and use the estimate from the present Lemma \ref{Cbetabounds} instead of the one from \cite[Lemma 2.9]{RS-Dir}.
\end{proof}

\section*{Acknowledgements}

The authors thank Xavier Cabr\'e for his guidance and useful discussions on the topic of this paper.

We also thank Louis Dupaigne and Manel Sanch\'on for interesting discussions on the topic of this paper.

\end{document}